\documentclass[11pt]{article}
\usepackage{amsmath}
\usepackage{amsthm}
\usepackage{latexsym}
\usepackage{tikz}
\usetikzlibrary{automata,positioning}
\usetikzlibrary{chains,fit,shapes}
\usetikzlibrary{calc}
\usetikzlibrary{arrows}
\usepackage{time}             
\usepackage{graphicx, epsfig}
\usepackage[T1]{fontenc}      
\usepackage{amssymb,amsmath}  
\usepackage{palatino}         
\usepackage{multimedia}
\usepackage{subfigure}
\usepackage{mathrsfs}
\usepackage[usenames,dvipsnames]{pstricks}
 \usepackage{epsfig}
 \usepackage{pst-grad} 
\usepackage{pst-plot} 
\usetikzlibrary{fit,matrix}
\tikzset{
mN/.style = {
    draw=#1, semithick, inner sep=0pt}
             }

\newtheorem{theorem}{Theorem}[section]
\newtheorem{definition}[theorem]{Definition}
\newtheorem{lemma}[theorem]{Lemma}
\newtheorem{remark}[theorem]{Remark}

\newtheorem{corollary}[theorem]{Corollary}
\newtheorem{example}[theorem]{Example}
\newtheorem{conjecture}{Conjecture}

\def\eref#1{(\ref{#1})}

\begin{document}

\title{Diameter of io-decomposable Riordan graphs of the Bell type\thanks{This work was supported by the Postdoctoral Research Program of Sungkyunkwan University (2016).}}
\date{}
\author{Ji-Hwan
Jung\\
{\footnotesize \textit{Applied Algebra and Optimization Research
Center, Sungkyunkwan University, Suwon 16419, Rep. of Korea}}
\\
{\footnotesize jh56k@skku.edu}
}

\maketitle

\begin{abstract}
Recently, in the paper \cite{CJKM1} we suggested the two conjectures
about the diameter of io-decomposable Riordan graphs of the Bell
type. In this paper, we give a counterexample for the first
conjecture. Then we prove that the first conjecture is true for the
graphs of some particular size and propose a new conjecture.
Finally, we show that the second conjecture is true for some special
io-decomposable Riordan graphs.
\end{abstract}



\section{Introduction}

 Let $\kappa[[z]]$ be the ring of formal power series in the variable $z$ over an
integral domain $\kappa$. A  Riordan matrix \cite{SGWW}
$L=[\ell_{n,k}]_{n,k\ge0}$ is defined by a pair of formal power
series $(g,f)\in \kappa[[z]]\times \kappa[[z]]$ with $f(0)=0$ such
that $[z^n]gf^k=\ell_{n,k}$ for $k\ge 0$ where $[z^n]$ is the
coefficient extraction operator. Usually, the Riordan matrix is
denoted by $L=(g,f)$ and its {\em leading principal submatrix} of
order $n$ is denoted by $(g, f)_n$. Since $f(0)=0$, every Riordan
matrix $(g,f)$ is an infinite lower triangular matrix. Most studies
on the Riordan matrices were related to combinatorics
\cite[etc.]{CKS,KS,MerSpr,Sprugnoli} or algebraic structures
\cite[etc.]{CHK,CJ,CJKS,CLMPS}.

Throughout this paper, we write $a\equiv b$ for $a\equiv b\;({\rm
mod\;2})$.

Recently, we in \cite{CJKM1,CJKM2} introduced a Riordan graph by
using the notion of the Riordan matrix modulo 2 as follows.
\begin{definition}{\rm A simple {\em labelled} graph $G$ with $n$ vertices is a {\em
Riordan graph} of order $n$ if the adjacency matrix ${\cal
A}(G)=[r_{i,j}]_{1\le i,j\le n}$ of $G$ is an $n\times n$ symmetric
$(0,1)$-matrix given~by
\begin{align*}
{\cal A}(G)\equiv(zg,f)_n+(zg, f)_n^T,\; \textrm{i.e.}\;
r_{i,j}=r_{j,i}\equiv\left\{
\begin{array}{ll}
[z^{i-2}]gf^{j-1}
 & \text{if $i>j$} \\
0 & \text{if $i=j$}
\end{array}
\right.
 \end{align*}
for some Riordan matrix $(g,f)$ over $\mathbb{Z}$. We denote such
$G$ by $G_n(g,f)$. In particular, the Riordan graph $G_n(g,f)$ is
called {\em proper} if $[z^0]g=[z^1]f=1$.}
\end{definition}

For example, consider the {\em Catalan graph} $CG_n=G_n(C(z),zC(z))$
where $C(z)$ is the generating function for the Catalan numbers,
i.e.
\begin{align}\label{Catalan-gf}
C(z)={1-\sqrt{1-4z}\over 2z}=\sum_{n\geq 0}\frac{1}{n+1}{2n\choose
n}z^n=1+z+2z^2+5z^3+14z^4+\cdots.
\end{align}
When $n=6$ we have Figure \ref{Catalan graph of order 6}.

  \begin{figure}
  \begin{center}
  \begin{tabular}{cc}
 \scalebox{1} 
 {
 \begin{pspicture}(0,0)(3.02625,0.0939063)
 \psdots[dotsize=0.12](0.08,1.2104688)
 \psdots[dotsize=0.12](1.48,1.2104688)
 \psdots[dotsize=0.12](2.88,1.2104688)
 \psdots[dotsize=0.12](0.38,-0.78953123)
 \psdots[dotsize=0.12](1.48,-1.3895313)
 \psdots[dotsize=0.12](2.58,-0.78953123)
 \usefont{T1}{ptm}{m}{n}
 \rput(0.07859375,1.5204687){2}
 \usefont{T1}{ptm}{m}{n}
 \rput(1.5009375,1.5204687){4}
 \usefont{T1}{ptm}{m}{n}
 \rput(2.8953125,1.5204687){6}
 \usefont{T1}{ptm}{m}{n}
 \rput(0.166875,-0.93953127){1}
 \usefont{T1}{ptm}{m}{n}
 \rput(1.3076563,-1.5395312){3}
 \usefont{T1}{ptm}{m}{n}
 \rput(2.6695313,-0.95953125){5}
 \psline[linewidth=0.04cm](0.08,1.2104688)(0.38,-0.78953123)
 \psline[linewidth=0.04cm](0.08,1.2304688)(1.48,-1.3695313)
 \psline[linewidth=0.04cm](0.08,1.2104688)(2.58,-0.78953123)
 \psline[linewidth=0.04cm](1.48,1.2104688)(1.48,-1.3495313)
 \psline[linewidth=0.04cm](1.48,1.2304688)(2.58,-0.74953127)
 \psline[linewidth=0.04cm](2.86,1.2104688)(1.48,-1.4095312)
 \psline[linewidth=0.04cm](2.88,1.2104688)(2.6,-0.76953125)
 \psline[linewidth=0.04cm](0.38,-0.78953123)(2.58,-0.78953123)
 \psline[linewidth=0.04cm](0.38,-0.78953123)(1.48,-1.3895313)
 \psline[linewidth=0.04cm](1.48,-1.3895313)(2.58,-0.78953123)
 \end{pspicture}
 }& $\mathcal{A}(CG_6)=\left(\begin{array}{cccccc}
        0 & 1 & 1 & 0 & 1 & 0 \\
        1 & 0 & 1 & 0 & 1 & 0 \\
        1 & 1 & 0 & 1 & 1 & 1 \\
        0 & 0 & 1 & 0 & 1 & 0 \\
        1 & 1 & 1 & 1 & 0 & 1 \\
        0 & 0 & 1 & 0 & 1 & 0
      \end{array}\right)$
  \end{tabular}
  \end{center}
\caption{The Catalan graph of order 6 and its adjacency
matrix}\label{Catalan graph of order 6}
\end{figure}
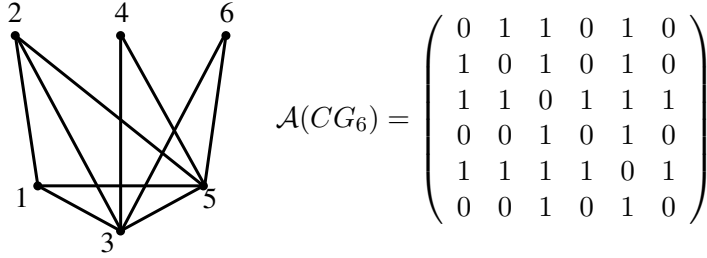

In \cite{CJKM1}, we studied the structural properties of families of
Riordan graphs obtained from infinite Riordan graphs, which include
a fundamental decomposition theorem and certain conditions on
Riordan graphs to have an Eulerian trail/cycle or a Hamiltonian
cycle. A Riordan graph $G_n(g,f)$ is called {\em Bell type} if
$f=zg$. Moreover, we studied the following Riordan graphs of special
Bell type.


\begin{definition}\cite{CJKM1}\label{io/ie-decomposable-def} {\rm Let $G_n=G_{n}(g,f)$ be a proper Riordan graph
 with the odd and even vertex sets $V_{o}=\{i\in V(G_n)\;|\;i\equiv 1\}$ and $V_{e}=\{i\in V(G_n)\;|\;i\equiv 0\}$,
 respectively. The graph $G_n$ is said to be {\em
io-decomposable} if $\left< V_{o}\right> \cong G_{\lceil
n/2\rceil}(g,f)$ and $\left< V_{e}\right>$ is a null
graph.}\end{definition}

A vertex in a graph $G$ is {\em universal} if it is adjacent to all
other vertices in $G$. The {\em distance} between two vertices $u,v$
in a graph $G$ is the number of edges in a shortest path between $u$
and $v$. The {\em diameter} of $G$ is the maximum distance between
all pairs of vertices, and it is denoted by diam$(G)$.

We found several properties of an io-decomposable Riordan graph of
the Bell type as follows.

\begin{lemma}\cite{CJKM1}\label{properties-of-io-decomposable}
Let $G_n=G_{n}(g,zg)$ be an io-decomposable Riordan graph of the
Bell type. Then we have the following.
\begin{itemize}
\item[(i)] If $n=2^{k}+1$ for
$k\geq 0$, then $G_{n}$ and $G_{n+1}$ have at least one universal
vertex, namely $2^k+1$.
\item[(ii)] $G_{n}$ is a
$(\lceil\log_2n\rceil+1)$-partite graph.
\item[(iii)] The chromatic number and the clique number of $G_n$ are $\lceil\log_2n\rceil+1$.
\item[(iv)] The diameter of $G_n$ is bounded by $\mathrm{diam}(G_n)\le \lfloor{\rm log}_2n\rfloor$.
In particular, if $n=2^k+2$ or $2^{k+1}+1$, for $k\ge1$, then
$\mathrm{diam}(G_n)=2$.
\item[(v)] If $2^k+1<n<2^{k+1}$ then $\mathrm{diam}(G_n)\le \lfloor{\rm
log}_2(n-2^k)\rfloor+1.$
\end{itemize}
\end{lemma}

A graph is called {\em weakly perfect} if its chromatic number
equals its clique number. By (iii) of Lemma
\ref{properties-of-io-decomposable}, every io-decomposable Riordan
graph of the Bell type is weakly perfect. It is known \cite{KPR}
that almost all $K_k$-free graphs are $(k-1)$-partite for $k\ge3$.
By (ii) and (iii) of Lemma \ref{properties-of-io-decomposable},
every io-decomposable Riordan graph $G_n(g,zg)$ of the Bell type is
$K_{\lceil\log_2n\rceil+2}$-free and
($\lceil\log_2n\rceil+1$)-partite for $n\ge2$. Thus the
io-decomposable Riordan graph of the Bell type is very interesting
object in Riordan graph theory.

It is known \cite{CJKM1} that the Pascal graph
$PG_n=G_n(1/(1-z),z/(1-z))$ and the Catalan graph
$CG_n=G_n(C(z),zC(z))$ are the io-decomposable Riordan graphs of the
Bell type. The following two conjectures introduced in \cite{CJKM1}
show significance of the Pascal graph $PG_n$ and the Catalan graph
$CG_n$.

\begin{conjecture}\cite{CJKM1}\label{conj1} {\rm Let $G_n$ be an
io-decomposable Riordan graph of the Bell type. Then
\begin{align*}
2={\rm diam}(PG_n)\le {\rm diam}(G_n) \le {\rm diam}(CG_n)
\end{align*}
for $n\geq 4$. Moreover, $PG_n$ is the only graph in the class of
io-decomposable graphs of the Bell type  whose diameter is $2$ for
{\em all} $n\ge4$.}
\end{conjecture}

\begin{conjecture}\cite{CJKM1}\label{conj2} {\rm We have that diam$(CG_{2^k})=k$ and there are no io-decomposable Riordan graphs $G_{2^k}\not\cong CG_{2^k}$ of the Bell type satisfying diam$(G_{2^k})=k$ for all $k\ge
1$.} \end{conjecture}

We note that ${\rm diam}(PG_n)=1$ if $n=2,3$ and ${\rm
diam}(PG_n)=2$ if $n\ge4$ since the vertex 1 is adjacent to all
other vertices, $PG_n\cong K_n$ if $n=2,3$ and $PG_n\not\cong K_n$
if $n\ge4$.

In this paper, we first give a counterexample of the upper bound in
Conjecture~\ref{conj1}. Then we prove that the upper bound in
Conjecture~\ref{conj1} is true for the graph of some particular size
and we propose a new conjecture for an upper bound of the diameter
of an io-decomposable Riordan graph of the Bell type. Finally, we
show that Conjecture~\ref{conj2} is true for some special
io-decomposable Riordan graphs.

\section{Upper bound of Conjecture 1}

It is known \cite{MRSV} that an infinite lower triangular matrix
$L=[\ell_{i,j}]_{i,j\ge0}$ with $\ell_{0,0}\ne0$ is a proper Riordan
matrix if and only if there is a unique sequence $(a_0,a_1,\ldots)$
with $a_0\neq0$ such that, for $i\ge j\ge0$,
\begin{align*}
\ell_{i+1,j+1}=a_0\ell_{i,j}+a_1\ell_{i,j+1}+\cdots+a_{i-j}\ell_{i,i}.
\end{align*}
The sequence $(a_i)_{i\ge0}$ is called  the {\it $A$-sequence} of
the Riordan array. Also, if $L=(g,f)$ then
\begin{align}\label{e:eq12}
f=zA(f),\quad{\rm or\ equivalently}\quad A=z/\bar f
\end{align}
where $A=\sum_{i\ge0}a_iz^i$ is the generating function for the
$A$-sequence of $(g,f)$. In particular, if $L$ is a binary Riordan
matrix $L\equiv(g,f)$ with $f'(0)=1$ then the sequence is called the
{\it binary $A$-sequence} $(1,a_1,a_2,\ldots)$ where
$a_k\in\{0,1\}$.

We in \cite{CJKM1} characterized the io-decomposable Riordan graph
$G_n(g,zg)$ of the Bell type, see the following lemma.

\begin{lemma}\cite{CJKM1}\label{charac-io-decomp}
Let $G_n=G_{n}(g,zg)$ be a Riordan graph of the Bell type. Then
$G_n$ is io-decomposable if and only if the binary $A$-sequence of
$(g,zg)$ is $(1,1,a_2,a_2,a_4,a_4,\ldots)$ where $a_{2j}\in\{0,1\}$
for all $j\ge 1$.
\end{lemma}

 Let $G_n$ be the io-decomposable Riordan
graph of the Bell type with its binary $A$-sequence generating
function $A(z)=\sum_{i=0}^{15}z^i$. By Lemma \ref{charac-io-decomp},
the Riordan graph $G_n$ is io-decomposable. By using the sage, we
compare the diameters between $CG_n$ and $G_n$ up to degree $n=100$.
Then we obtain the following 13 counterexamples for the upper bound
of Conjecture \ref{conj1}.
\begin{center}
\begin{tabular}{c|c|c|c|c|c|c|c|c|c|c|c|c|c}

  $n$          & 44 & 45 & 46 & 47 & 48 & 78 & 79 & 80 & 87 & 88 & 89 & 90 & 91 \\\hline
  diam$(CG_n)$ &  3 &  3 &  3 &  3 &  3 &  3 &  3 &  3 &  3 &  3 &  3 &  3 &  3 \\\hline
  diam$(G_n)$  &  4 &  4 &  4 &  4 &  4 &  4 &  4 &  4 &  4 &  4 &  4 &  4 &  4
\end{tabular}
\end{center}

If for a Riordan graph $G_n(g,f)$ with $[z^1]f=1$, the relabelling
is done by {\em reversing} the vertices in $[n]$, that is, by
replacing a label $i$ by $n+1-i$ for each $i\in[n]$, then the
resulting graph will always be a Riordan graph given by the
following lemma. We denote the reverse relabelling of $G_n$ by
$G_n^r$.

\begin{lemma}\cite{CJKM1}\label{e:reverse}
The reverse relabelling of a Riordan graph $G_n(g,f)$ with $f'(0)=1$
is the Riordan graph
\begin{align*}
G_n^r(g,f)=G_n(g(\bar f)\cdot (\bar f)'\cdot(z/\bar f)^{n-1},\bar f)
\end{align*}
where $\bar f$ is the compositional inverse of $f$.
\end{lemma}

Now, we prove that the upper bound in Conjecture \ref{conj1} is true
if $n=2^k$, $n=2^k-1$ or $n=1+2^m+2^k$ where $k\ge2$ and $1\le m<k$.

\begin{lemma}\label{reverse-relabelling-A-sequence-io}
Let $G_n=G_n(g,zg)$ be a proper Riordan graph and $A(z)$ be the
generating function for its binary $A$-sequence. Then the reverse
relabelling of $G_n$ is the Riordan graph given by
\begin{align*}
G_n^r=G_n((zA(z))'\cdot A^{n-2}(z),z/A(z)).
\end{align*}
In particular, if $G_n$ is an io-decomposable Riordan graph of the
Bell type then the reverse relabelling of $G_n$ is the Riordan graph
given by
\begin{align*}
G_n^r=G_n(A'(z)\cdot A^{n-2}(z),z/A(z)).
\end{align*}
\end{lemma}
\begin{proof} Let $f=zg$ and $\bar f$ be the compositional inverse
of $f$. Since \eref{e:eq12} leads to
\begin{align*}
g(\bar f)=z/\bar f=A(z)\quad {\rm and}\quad (\bar f)'=\left(z\over
A(z)\right)'\equiv{A(z)+zA'(z)\over A^2(z)}=(zA(z))'A^{-2}(z),
\end{align*}
we obtain
\begin{align*}
g(\bar f)\cdot (\bar f)'\cdot(z/\bar f)^{n-1}\equiv(zA(z))'\cdot
A^{n-2}(z)\quad {\rm and}\quad \bar f=z/A(z).
\end{align*}
Thus, by Lemma \ref{e:reverse}, we obtain the desired result. In
particular, if $G_n$ is an io-decomposable Riordan graph of the Bell
type then it follows from Lemma \ref{charac-io-decomp} that
$(zA(z))'\equiv A'(z)$. Hence the proof follows.
\end{proof}

If the base $p$ (a prime) expansion of $n$ and $m$ is
$n=n_0+n_1p+n_2p^2+\cdots$ and $m=m_0+m_1p+m_2p^2+\cdots$
respectively then
\begin{eqnarray*}
{n\choose m}\equiv\prod_i{n_i\choose m_i}\; ({\rm mod}\; p).
\end{eqnarray*}
This is called the Lucas's theorem.

Let $G_n$ be an io-decomposable Riordan graph of the Bell type.
Since diam$(G_{n})\le k$ if $n=2^k$ with $k\ge0$ by (iv) of Lemma
\ref{properties-of-io-decomposable}, the following theorem shows
that the upper bound of Conjecture \ref{conj1} is true if $n=2^k$
for $k\ge1$. We denote the distance between two vertices $u,v$ in a
graph $G$ by $d_G(u,v)$.

\begin{theorem}\label{diam-Catalan}
For an integer $k\ge1$, we obtain
\begin{align*}
{\rm diam}(CG_{2^k})=k.
\end{align*}
\end{theorem}
\begin{proof} First we show that
$CG_{2^k}^r=G_{2^k}(1,z+z^2)$ for $k\ge1$. It follows from
\eref{Catalan-gf} and \eref{e:eq12} that the generating function for
$A$-sequence of $(C,zC)$ is ${1\over 1-z}$. By Lemma
\ref{reverse-relabelling-A-sequence-io}, we obtain
\begin{align*}
A'(z)\cdot A^{2^k-2}(z)&=\left({1\over
1-z}\right)^2\cdot\left({1\over 1-z}\right)^{2^k-2}=\left({1\over
1-z}\right)^{2^k}
\end{align*}
so that by Lemma \ref{e:reverse} the reverse relabelling of the
Catalan graph is
\begin{align}\label{reverse-relabelling-CG}
CG_{2^k}^r=G_{2^k}((1-z)^{-2^k},z+z^2).
\end{align}
Since $(1-z)^{-2^k}=\sum_{j\ge0}{2^k+j-1\choose 2^k-1}z^j$, by
Lucas's theorem we obtain
\begin{align*}
{2^k+j-1\choose 2^k-1}\equiv 1\;\;\textrm{for $j=0$} \quad{\rm
and}\quad {2^k+j-1\choose 2^k-1}\equiv0\;\;\textrm{for $1\le j\le
2^k-1$}
\end{align*}
which imply $G_{2^k}((1-z)^{-2^k},z+z^2)=G_{2^k}(1,z+z^2)$. Thus, by
\eref{reverse-relabelling-CG}, $CG_{2^k}^r=G_{2^k}(1,z+z^2)$. Let
$m_i={\rm max}\{j\in V(CG_{2^k}^r)\;|\;ij\in E(CG_{2^k}^r)\}.$ For
each $i\in\{1,2,\ldots,2^{k-1}\}$, we obtain
\begin{align}\label{equation}
m_i={\rm max}\{j\in
V(CG_{2^k}^r)\;|\;[z^{j-2}](z+z^2)^{i-1}\equiv1\}=2i.
\end{align}
Since $m_1<m_2<\cdots<m_{2^{k-1}}$, a unique shortest path from 1 to
$2^k$ in $CG_{2^k}^r$ is
$2^0\rightarrow2^1\rightarrow\cdots\rightarrow2^{k-1}\rightarrow2^k$
so that $d_{CG_{2^k}^r}(1,2^k)=k$. Hence, by (iv) of Lemma
\ref{properties-of-io-decomposable}, we obtain the desired result.
\end{proof}

By Lemma \ref{charac-io-decomp}, the following lemma is obtained
from \cite[Theorem 3.7]{CJKM1} when $\ell=0$.

\begin{lemma}\label{fractal-Riordan-graph}
Let $G_n$ be an io-decomposable Riordan graph of the Bell type and
$G=\lim_{n\rightarrow\infty}G_n$. For each $s\ge 0$, $G$ has the
following fractal properties:
\begin{itemize}
\item[{\rm(i)}] $G_{2^s+1}=\left<\{1,2,\ldots,2^{s}+1\}\right>\cong
\left<\{\alpha2^{s}+1,\alpha2^{s}+2,\ldots,(\alpha+1)2^{s}+1\}\right>$;
\item[{\rm(ii)}]$G_{2^s}=\left<\{1,2,\ldots,2^{s}\}\right>\cong
\left<\{\alpha2^{s}+1,\alpha2^{s}+2,\ldots,(\alpha+1)2^{s}\}\right>$
\end{itemize}
where $\alpha\ge1$.
\end{lemma}

We can ask that how many vertex pairs $(u,v)$ can have the maximal
distance $k$ in $CG_{2^k}$. By using Lemma
\ref{fractal-Riordan-graph}, the answer is given by the following
theorem.

\begin{theorem}\label{Catalan-graph-distance-k}
Let $k\ge1$ be an integer. There exist exactly $2^{k-1}$ vertex
pairs $(i,2^k)$ with $i\in\{1,\ldots,2^{k-1}\}$ such that
$d_{CG_{2^k}}(i,2^k)=k$ is the maximal distance in $CG_{2^k}$.
\end{theorem}
\begin{proof} Since $CG_{2^k}^r$ is the reverse relabelling of $CG_{2^k}$, this
theorem is equivalent to the following:
\begin{itemize}
\item $d_{CG_{2^k}^r}(i,j)=k$ if $i=1$ and $j\in\{2^{k-1}+1,\ldots,2^k\}$;
\item $d_{CG_{2^k}^r}(i,j)\le k-1$ otherwise.
\end{itemize}
Since by (i) of Lemma \ref{properties-of-io-decomposable} the vertex
$2^{k-1}+1$ is adjacent to all vertices $1,\ldots,2^{k-1}$ in
$CG_{2^k}$, the vertex $2^{k-1}$ is adjacent to all vertices
$2^{k-1}+1,\ldots,2^k$ in $CG_{2^k}^r$. So by \eref{equation} the
shortest path from 1 to $j$ in $CG_{2^k}^r$ is
\begin{align*}
2^0\rightarrow2^1\rightarrow\cdots\rightarrow2^{k-1}\rightarrow
j\;\; {\rm where}\;\; j\in\{2^{k-1}+1,\ldots,2^k\}
\end{align*}
and thus $d_{CG_{2^k}^r}(i,j)=k$ if $i=1$ and
$j\in\{2^{k-1}+1,\ldots,2^k\}$.

Let $V_1=\{i\in V(CG_{2^k}^r)\;|\;1\le i\le 2^{k-1}\}$ and
$V_2=\{j\in V(CG_{2^k}^r)\;|\;2^{k-1}< j\le 2^k\}$. Since
$\left<V_1\right>\cong \left<V_2\right>\cong CG_{2^{k-1}}^r$ by
Lemma \ref{fractal-Riordan-graph} and $G_{2^{k-1}}\cong
CG_{2^{k-1}}$, it follows from Theorem \ref{diam-Catalan} that
\begin{align*}
d_{CG_{2^k}^r}(i,j)\le k-1\;\; \textrm{if $i,j\in V_1$ or $i,j\in
V_2$.}
\end{align*}
Now it is enough to show that $d_{CG_{2^k}^r}(i,j)<k$ if $i\in
V_1\backslash\{1\}$ and $j\in V_2$ for $k\ge2$. We prove this by
induction on $k\ge2$. Let $k=2$. Since the adjacency matrix of
$CG_{4}^r$ is given by
\begin{align*}
\mathcal{A}(CG_{4}^r)=\left(
                   \begin{array}{cccc}
                     0 & 1 & 0 & 0 \\
                     1 & 0 & 1 & 1 \\
                     0 & 1 & 0 & 1 \\
                     0 & 1 & 1 & 0 \\
                   \end{array}
                 \right),
\end{align*}
we see that $d_{CG_{4}^r}(2,3)=d_{CG_{4}^r}(2,4)=1<2$. Thus it holds
for $k=2$. Let $k\ge3$. Since $\left<V_1\right>\cong CG_{2^{k-1}}^r$
and the vertex $2^{k-1}$ is adjacent to all vertices $j\in V_2$ in
$CG_{2^k}^r$, we obtain
\begin{align*}
d_{CG_{2^k}^r}(i,j)&\le
d_{CG_{2^k}^r}(i,2^{k-1})+d_{CG_{2^k}^r}(2^{k-1},j)\le d_{CG_{2^{k-1}}^r}(i,2^{k-1})+1\\
&\le k-1\quad\textrm{(by induction)}
\end{align*}
where $i\in V_1\backslash\{1\}$ and $j\in V_2$. Hence the proof
follows.
\end{proof}

\begin{example}
{\rm Let us consider the Catalan graph $CG_8=G_8(C(z),zC(z))$ of
order 8. Since its reverse relabeling is $CG_8^r=G_8(1,z+z^2)$, we
obtain Figure \ref{Reverse relabeling of CG8} from the adjacency
matrix
\begin{align*}
\mathcal{A}(CG_8^r)=\left(
                      \begin{array}{cccccccc}
0 & 1 & 0 & 0 & 0 & 0 & 0 & 0 \\
1 & 0 & 1 & 1 & 0 & 0 & 0 & 0 \\
0 & 1 & 0 & 1 & 0 & 1 & 0 & 0 \\
0 & 1 & 1 & 0 & 1 & 1 & 1 & 1 \\
0 & 0 & 0 & 1 & 0 & 1 & 0 & 0 \\
0 & 0 & 1 & 1 & 1 & 0 & 1 & 1 \\
0 & 0 & 0 & 1 & 0 & 1 & 0 & 1 \\
0 & 0 & 0 & 1 & 0 & 1 & 1 & 0
\end{array}
                    \right).
\end{align*}
Thus we can see that the four vertex pairs $(1,5),(1,6),(1,7)$ and
$(1,8)$ in $CG_8^r$ have maximal distance 3 i.e., the four vertex
pairs $(8,4),(8,3),(8,2)$ and $(8,1)$ in $CG_8$ have the maximal
distance 3.

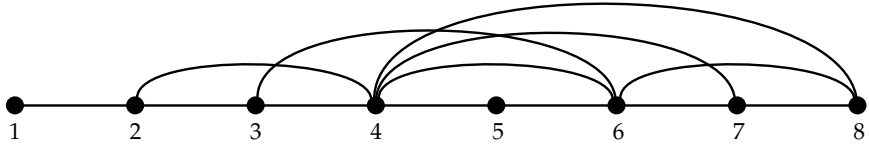
\begin{figure}
\begin{center}
\psscalebox{0.8 0.8} 
{
\begin{pspicture}(0,-1.1889503)(14.295673,1.1889503)
\psbezier[linecolor=black,
linewidth=0.04](4.1786056,-0.417541)(4.1786056,1.1516898)(10.116906,1.0286129)(10.132452,-0.4790794607309317)
\psbezier[linecolor=black,
linewidth=0.04](6.1170673,-0.44831023)(6.1170673,1.798543)(14.132291,1.6027387)(14.147837,-0.5560025376539983)
\psbezier[linecolor=black,
linewidth=0.04](6.1478367,-0.47907946)(6.1478367,1.0901513)(12.101441,1.0132282)(12.117067,-0.49446407611554605)
\psbezier[linecolor=black,
linewidth=0.04](6.19399,-0.417541)(6.19399,0.3670744)(10.070913,0.3055359)(10.086298,-0.44831022996169395)
\psline[linecolor=black,
linewidth=0.04](6.209375,-0.54061794)(8.132452,-0.54061794)
\psbezier[linecolor=black,
linewidth=0.04](2.1786058,-0.417541)(2.1786058,0.3670744)(6.0555286,0.3055359)(6.0709133,-0.44831022996169395)
\psline[linecolor=black,
linewidth=0.04](0.209375,-0.54061794)(2.009375,-0.54061794)
\psdots[linecolor=black, dotsize=0.3](0.14783654,-0.54061794)
\psdots[linecolor=black, dotsize=0.3](2.1478364,-0.54061794)
\psdots[linecolor=black, dotsize=0.3](4.1478367,-0.54061794)
\psdots[linecolor=black, dotsize=0.3](6.1478367,-0.54061794)
\psdots[linecolor=black, dotsize=0.3](8.147837,-0.54061794)
\psdots[linecolor=black, dotsize=0.3](10.147837,-0.54061794)
\psdots[linecolor=black, dotsize=0.3](12.147837,-0.54061794)
\psdots[linecolor=black, dotsize=0.3](14.147837,-0.54061794)
\rput[bl](0.055528842,-1.1){1} \rput[bl](2.0555289,-1.1){2}
\rput[bl](4.0555286,-1.1){3} \rput[bl](6.0555286,-1.1){4}
\rput[bl](8.086298,-1.1){5} \rput[bl](10.086298,-1.1){6}
\rput[bl](12.086298,-1.1){7} \rput[bl](14.086298,-1.1){8}
\psline[linecolor=black,
linewidth=0.04](2.1632211,-0.54061794)(6.0401444,-0.54061794)
\psbezier[linecolor=black,
linewidth=0.04](10.209375,-0.417541)(10.209375,0.3670744)(14.086298,0.3055359)(14.101683,-0.44831022996169395)
\psline[linecolor=black,
linewidth=0.04](8.193991,-0.54061794)(14.024759,-0.54061794)
\end{pspicture}
}
\end{center}
\caption{The graph of $CG_8^r=G_n(1,z+z^2)$}\label{Reverse
relabeling of CG8}
\end{figure}}
\end{example}

Let $G_n$ be an io-decomposable Riordan graph of the Bell type.
Since it follows from (iv) of Lemma
\ref{properties-of-io-decomposable} that diam$(G_{n})\le k-1$ if
$n=2^k-1$, the following corollary shows that the upper bound of
Conjecture \ref{conj1} is true if $n=2^k-1$ for $k\ge1$.

\begin{corollary}\label{corollary-catalan-diameter}
For an integer $k\ge1$, we obtain
\begin{align*}
{\rm diam}(CG_{2^k-1})=k-1.
\end{align*}
\end{corollary}
\begin{proof}
By Theorem \ref{Catalan-graph-distance-k}, we obtain
diam$(CG_{2^k-1})\le k-1$. It follows from Lemma
\ref{reverse-relabelling-A-sequence-io} that one can show
$CG_{2^k-1}^r=G_n(1+z,z+z^2)$. By using the similar proof in Theorem
\ref{diam-Catalan}, we can show that $2^1-1\rightarrow
2^2-1\rightarrow\cdots\rightarrow2^k-1$ is the shortest path from 1
to $2^k-1$ in $CG_{2^k-1}^r$, i.e. $d_{CG_{2^k-1}^r}(1,2^k-1)=k-1$.
Since it follows from Theorem \ref{Catalan-graph-distance-k} that
diam$(CG_{2^k-1})\le k-1$, we obtain diam$(CG_{2^k-1})=k-1$. Hence
the proof follows.
\end{proof}

The following lemma is useful to obtain Theorem
\ref{upperbound-conj1} and Conjecture \ref{conj3}.

\begin{lemma}\label{diam-Catalan-1em}
Let $n=1+2^m+\sum_{j=0}^s2^{k+j}$ be an integer with $k>m\ge1$. If
$G_n$ be the io-decomposable Riordan graph of the Bell type, then we
obtain
\begin{align*}
{\rm diam}(G_n)\le\left\{
\begin{array}{ll}
s+2
 & \text{if $m=1$;} \\
s+3 & \text{otherwise.}%
\end{array}
\right.
\end{align*}
\end{lemma}
\begin{proof} We prove this by induction on $s\ge0$. Let $s=0$, i.e. $n=1+2^m+2^k$. If $m=1$ then it follows
from (v) of Lemma \ref{properties-of-io-decomposable} that ${\rm
diam}(G_{2^k+3})=2$. For $k>m\ge2$, let $V_1=\{i\in V(G_n)\;|\;1\le
i\le 2^{k}+1\}$ and $V_2=\{j\in V(G_n)\;|\;2^{k}+1\le j\le n\}$.
Since $\left<V_1\right>\cong G_{2^{k}+1}$ and $\left<V_2\right>\cong
G_{2^{m}+1}$ by Lemma \ref{fractal-Riordan-graph}, it follows from
(iv) of Lemma~\ref{properties-of-io-decomposable} that ${\rm
diam}(\left<V_1\right>)={\rm diam}(\left<V_2\right>)=2$. Let $i\in
V_1\backslash\{2^{k}+1\}$ and $j\in V_2\backslash\{2^{k}+1\}$. Now
it is enough to show that $d(i,j)\le 3$. Since the vertices
$2^{k}+1$ and $2^k+2^m+1$ are the universal vertices in
$\left<V_1\right>$ and $\left<V_2\right>$ respectively, we obtain
\begin{align*}
d_{G_n}(i,j)&\le
d_{G_n}(i,2^{k}+1)+d_{G_n}(2^{k}+1,2^k+2^m+1)+d_{G_n}(2^k+2^m+1,j)\\
&\le
d_{\left<V_1\right>}(i,2^{k}+1)+d_{\left<V_2\right>}(2^{k}+1,2^k+2^m+1)+d_{\left<V_2\right>}(2^k+2^m+1,j)\\
&\le3.
\end{align*}
Thus the theorem holds for $s=0$.

Let $s\ge1$, i.e. $n=1+2^m+\sum_{j=0}^s2^{k+j}$. For $k>m\ge1$, let
$W_1=\{i\in V(G_n)\;|\;1\le i\le 2^{k+s}+1\}$ and $W_2=\{j\in
V(G_n)\;|\;2^{k+s}+1\le j\le n\}$. Since by Lemma
\ref{fractal-Riordan-graph} we obtain $\left<W_1\right>\cong
G_{2^{k+s}+1}$ and $\left<W_2\right>\cong G_{n-2^{k+s}}$ , by (iv)
of Lemma~\ref{properties-of-io-decomposable} we obtain ${\rm
diam}(\left<W_1\right>)=2$ and  by induction we obtain ${\rm
diam}(\left<W_2\right>)\le s+1$ if $m=1$ or ${\rm
diam}(\left<W_2\right>)\le s+2$ if $k>m>1$. Let $i\in
W_1\backslash\{2^{k+s}+1\}$ and $j\in W_2\backslash\{2^{k+s}+1\}$.
Now it is enough to show that $d_{G_n}(i,j)\le s+2$ if $m=1$ or
$d_{G_n}(i,j)\le s+3$ if $k>m>1$. Since the vertices $2^{k+1}+1$ are
the universal vertices in $\left<W_1\right>$, we obtain
\begin{align*}
d_{G_n}(i,j)&\le d_{G_n}(i,2^{k+s}+1)+d_{G_n}(2^{k+s}+1,j)\le
1+d_{G_{n-2^s}}(2^{k+s}+1,j).
\end{align*}
Hence, by induction, we obtain the desired result.
\end{proof}

From Lemma \ref{diam-Catalan-1em}, the following theorem shows that
the upper bound of Conjecture \ref{conj1} is true if $n=1+2^m+2^k$
for $k>m\ge1$.

\begin{theorem}\label{upperbound-conj1}
Let $k$ and $m$ be integers with $k>m\ge1$. Then
\begin{align*}
{\rm diam}(CG_{1+2^m+2^k})=\left\{
\begin{array}{ll}
2
 & \text{if $m=1$;} \\
3 & \text{otherwise.}%
\end{array}
\right.
\end{align*}
\end{theorem}
\begin{proof} Since by Lemma \ref{diam-Catalan-1em} we obtain ${\rm diam}(CG_{2^k+3})=2$,  it is enough to show that
diam$(CG_{1+2^m+2^k})=3$ for $k>m>1$. Now let $k$ and $m$ be
integers with $k>m>1$. By Lemma \ref{e:reverse}, the reverse
relabelling of the Catalan graph $CG_{2^k+2^m+1}$ is
\begin{align}\label{reverse-relabelling-CG-1}
CG_{1+2^m+2^k}^r=G_{1+2^m+2^k}((1-z)^{-1-2^m-2^k},z+z^2).
\end{align}
Let $\mathcal{A}(CG_{1+2^m+2^k})=[c_{i,j}]$ and
$\mathcal{A}(CG_{1+2^m+2^k}^r)=[r_{i,j}]$. By
\eref{reverse-relabelling-CG-1}, we obtain
\begin{align}\label{eq1}
c_{2^m+2^k,j}=r_{2^m+2^k+2-j,2}\equiv\left\{
\begin{array}{ll}
1
 & \text{if $j=2^k+2^m+1$;} \\
 0
 & \text{if $j=2^k+2^m$;} \\
\text{$[z^{2^k+2^m-j}]z(1-z)^{-2^k-2^m}$} & \text{otherwise.}%
\end{array}
\right.
\end{align}
Since
\begin{align*}
[z^{2^k+2^m-j}]z(1-z)^{-2^k-2^m}={2^{k+1}+2^{m+1}-j-2\choose
2^k+2^m-1},
\end{align*}
by Lucas's theorem we obtain for $j=1,\ldots,2^k+2^m-1$
\begin{align}\label{eq2}
c_{2^k+2^m,j}&\equiv{2^{k+1}+2^{m+1}-j-2\choose
2^k+2^m-1}\nonumber\\
&\equiv\left\{
\begin{array}{ll}
1
 & \text{if $j\in\{2^{m+1}+t2^m-1\;|\;t=0,\ldots,2^{k-m}-1\}$;} \\
 0
 & \text{otherwise.}
\end{array}
\right.
\end{align}
By \eref{eq1} and \eref{eq2}, the set $N(2^k+2^m)$ of neighbors of
the vertex $2^k+2^m$ in $CG_{2^k+2^m+1}$ is
\begin{align*}
N(2^k+2^m)=\{2^{m+1}+t2^m-1\;|\;t=0,\ldots,2^{k-m}-1\}\cup\{2^k+2^m+1\}.
\end{align*}
It is known \cite{DS} that $[z^n]C(z)\equiv 1$ if and only if
$n=2^k-1$ for $k\ge1$. It implies
\begin{align*}
c_{i,1}=\left\{\begin{array}{ll} 1
 & \text{if $j\in\{2^s+1\;|\;s=0,1,\ldots,k\}$;} \\
 0
 & \text{otherwise.}
\end{array}
\right.
\end{align*}
Thus the set $N(1)$ of neighbors of the vertex $1$ in
$CG_{2^k+2^m+1}$ is
\begin{align*}
N(1)=\{2^{s}+1\;|\;s=0,\ldots,k\}.
\end{align*}
Since $$2^k+2^m\not\in N(1),\;\;1\not\in N(2^k+2^m)\;\;{\rm
and}\;\;N(1)\cap N(2^k+2^m)=\emptyset,$$ the distance between
vertices $1$ and $2^k+2^m$ in $CG_{1+2^m+2^k}$ is at least 3 so that
by Lemma \ref{diam-Catalan-1em} we obtain diam$(CG_{1+2^m+2^k})=3$.
Hence the proof follows.
\end{proof}

We end this section with the following conjecture.
\begin{conjecture}\label{conj3}
Let $n=1+2^m+\sum_{j=0}^s2^{k+j}$ be an integer with $k>m\ge1$ and
$s\ge1$. Then
\begin{align*}
{\rm diam}(CG_n)=\left\{
\begin{array}{ll}
s+2
 & \text{if $m=1$;} \\
s+3 & \text{otherwise.}%
\end{array}
\right.
\end{align*}
\end{conjecture}
\begin{remark}
If Conjecture \ref{conj3} is true, then by Lemma
\ref{diam-Catalan-1em} the upper bound of Conjecture \ref{conj1} is
true if $n=1+2^m+\sum_{j=0}^s2^{k+j}$ for $k>m\ge1$ and $s\ge1$. By
using the sage, we have checked that Conjecture \ref{conj3} is true
for $n\le2^8$.
\end{remark}

\section{Conjecture 2}

In this section, we show that Conjecture \ref{conj2} is true for
some special io-decomposable Riordan graphs of the Bell type.

\begin{lemma}\label{lemma-diameter-inequality}
Let $G_n=G_n(g,zg)$ be an io-decomposable Riordan graph. If there
exists $k\ge2$ such that diam$(G_{2^k})=s$ then
diam$(G_{2^{k+m}})\le s+m$ for all $m\ge1$.
\end{lemma}
\begin{proof}
Let $V_1=\{i\in V(G_n)\;|\;1\le i\le 2^{k+m-1}+1\}$ and $V_2=\{j\in
V(G_n)\;|\;2^{k+m-1}+1\le j\le 2^{k+m}\}$ be the vertex subsets of
$V(G_{2^{k+m}})$. Since $\left<V_1\right>\cong G_{2^{k+m-1}+1}$ has
a universal vertex $2^{k+m-1}+1$ and by Lemma
\ref{fractal-Riordan-graph} we obtain $\left<V_2\right>\cong
G_{2^{k+m-1}}$, we obtain
\begin{align}\label{diameter-inequality}
{\rm diam}(G_{2^{k+m}})\le {\rm diam}(\left<V_2\right>)+1= {\rm
diam}(G_{2^{k+m-1}})+1.
\end{align}
Let diam$(G_{2^k})=s$. Applying for $m=1$ in
\eref{diameter-inequality}, we obtain diam$(G_{2^{k+1}})\le s+1$.
Applying again for $m=2$ in \eref{diameter-inequality}, we obtain
diam$(G_{2^{k+2}})\le s+2$. By repeating this process, we obtain the
desired result.
\end{proof}

Let ${\cal B}(g,f)$ denote a binary Riordan matrix, i.e. ${\cal
B}(g,f)\equiv (g,f)$. We note that a Riordan matrix
$[b_{i,j}]_{i,j\geq 0}$ is of the Bell type given by ${\cal
B}(g,zg)$ with $g(0)=1$
 if and only if, for $i\ge j\ge0$,
\begin{align}\label{recurrence-relation-Bell}
 b_{i+1,0}&\equiv a_1b_{i,0}+a_2b_{i,1}+\cdots+a_{i+1}b_{i,i},\\
b_{i+1,j+1}&\equiv
b_{i,j}+a_1b_{i,j+1}+\cdots+a_{i-j}b_{i,i}\nonumber
\end{align}
where $(1,a_1,\ldots)$ is the binary $A$-sequence of ${\cal
B}(g,zg)$. Let $G_n=G_n(g,zg)$ and $\mathcal{A}(G_n)=[r_{i,j}]_{1\le
i,j\le n}$ where $g(0)=1$. Since $r_{i,j}=b_{i-2,j-1}$ for $i>j\ge
1$, by \eref{recurrence-relation-Bell} we need the finite term
$(1,a_1,\ldots,a_{n-2})$ of the binary $A$-sequence to determine
$\mathcal{A}(G_n)$.

\begin{theorem}
Let $G_{2^k}=G_{2^k}(g,zg)$ be an io-decomposable Riordan graph. If
the binary $A$-sequence of $(g,zg)$ is of the following form
\begin{align}\label{A-sequence-condition}
(\underbrace{1,1,\ldots,1}_{\textrm{$2^m-2$
copies}},0,0,a_{2^m},a_{2^m},a_{2^m+2},a_{2^m+2},\ldots),\;\;a_j\in\{0,1\},\;m\ge4
\end{align}
then for $k\ge4$ we obtain
\begin{align*}
\textrm{diam}(G_{2^k})<\textrm{diam}(CG_{2^k})=k.
\end{align*}
\end{theorem}
\begin{proof}
First we show that diam$(G_{2^m})=m-1$. Since the induced subgraph
$H$ of $\{1,2\ldots,2^m-1\}$ in $G_{2^m}$ is $H=CG_{2^m-1}$ and
$CG_{2^m-1}^r=G_{2^m-1}(1+z,z+z^2)$, the $(2^m-1)$th row of
$\mathcal{A}(G_{2^m})=[r_{i,j}]$ is given by
\begin{align}\label{2m-1th-row}
(0,\ldots,0,1,1,0,1)=(r_{2^m-1,i})_{i=1}^{2^m}.
\end{align}
By \eref{recurrence-relation-Bell}, \eref{A-sequence-condition} and
\eref{2m-1th-row}, the $2^m$th row in
$\mathcal{A}(G_{2^m})=[r_{i,j}]$ is given by
\begin{align*}
(1,0,\ldots,0,1,0)=(r_{2^m,i})_{i=1}^{2^m}
\end{align*}
which means the only two vertices 1 and $2^m-1$ are adjacent to the
vertex $2^m$ in $G_{2^m}$. Let $V_1=\{1,\ldots,2^{m-1}+1\}$ and
$V_2=\{2^{m-1}+1,\ldots,2^{m}-1\}$ be the vertex subsets of
$V(G_{2^m})$. Since $\left<V_1\right>$ has the universal vertex
$2^{m-1}+1$ and $\left<V_2\right>\cong CG_{2^{m-1}-1}$, if $v_1\in
V_1$ and $v_2\in V_2$ then we respectively obtain
$d_{G_{2^m}}(v_1,2^m)\le 3$ and
\begin{align*}
d_{G_{2^m}}(v_2,2^m)\le {\rm diam}(CG_{2^{m-1}-1})+1\le
2^m-1\quad(\textrm{by Corollary \ref{corollary-catalan-diameter}})
\end{align*}
which implies diam$(G_{2^m})=m-1$. Hence, by Lemma
\ref{lemma-diameter-inequality}, we obtain the desired result.
\end{proof}

\begin{center}
\begin{tabular}{c|c||c|c}
  $A$-seq. of $G_{8}$ & diam$(G_{8})$ & $A$-seq. of $G_{8}$ & diam$(G_{8})$\\
  \hline\hline
  $(1,1,0,0,0,0,0)$ & 2 & $(1,1,1,1,0,0,1)$ & 2 \\
  \hline
  $(1,1,1,1,0,0,0)$ & 2 & $(1,1,0,0,1,1,1)$ & 2 \\
  \hline
  $(1,1,0,0,1,1,0)$ & 2 & $(1,1,1,1,1,1,0)$ & 3 \\
  \hline
  $(1,1,0,0,0,0,1)$ & 2 & $(1,1,1,1,1,1,1)$ & 3 \\
  \hline
  $(1,1,1,1,1,1,0)$ & 2 &  &
\end{tabular}
\centerline{{\rm Table 1} \quad Diameters of io-decomposable Riordan
graphs of the Bell type with degree 8}
\end{center}

\begin{center}
{\small\begin{tabular}{c|c||c|c}
  $A$-seq. of $G_{16}$ & diam$(G_{16})$ & $A$-seq. of $G_{16}$ & diam$(G_{16})$\\
  \hline\hline
  $(1,1,1,1,1,1,0,0,0,0,0,0,0,0,0)$ & 3 & $(1,1,1,1,1,1,1,1,0,0,0,0,0,0,0)$  & 3 \\
  \hline
  $(1,1,1,1,1,1,0,0,1,1,0,0,0,0,0)$ & 3 & $(1,1,1,1,1,1,1,1,1,1,0,0,0,0,0)$  & 3 \\
  \hline
  $(1,1,1,1,1,1,0,0,0,0,1,1,0,0,0)$ & 3 & $(1,1,1,1,1,1,1,1,0,0,1,1,0,0,0)$  & 3 \\
  \hline
  $(1,1,1,1,1,1,0,0,0,0,0,0,1,1,0)$ & 3 & $(1,1,1,1,1,1,1,1,0,0,0,0,1,1,0)$  & 3 \\
  \hline
  $(1,1,1,1,1,1,0,0,0,0,0,0,0,0,1)$ & 3 & $(1,1,1,1,1,1,1,1,0,0,0,0,0,0,1)$  & 3 \\
  \hline
  $(1,1,1,1,1,1,0,0,1,1,1,1,0,0,0)$ & 3 & $(1,1,1,1,1,1,1,1,1,1,1,1,0,0,0)$  & 3 \\
  \hline
  $(1,1,1,1,1,1,0,0,1,1,0,0,1,1,0)$ & 3 & $(1,1,1,1,1,1,1,1,1,1,0,0,1,1,0)$  & 3 \\
  \hline
  $(1,1,1,1,1,1,0,0,1,1,0,0,0,0,1)$ & 3 & $(1,1,1,1,1,1,1,1,1,1,0,0,0,0,1)$  & 3 \\
  \hline
  $(1,1,1,1,1,1,0,0,0,0,1,1,1,1,0)$ & 3 & $(1,1,1,1,1,1,1,1,0,0,1,1,1,1,0)$  & 3 \\
  \hline
  $(1,1,1,1,1,1,0,0,0,0,1,1,0,0,1)$ & 3 & $(1,1,1,1,1,1,1,1,0,0,1,1,0,0,1)$  & 3 \\
  \hline
  $(1,1,1,1,1,1,0,0,1,1,1,1,1,1,0)$ & 3 & $(1,1,1,1,1,1,1,1,0,0,0,0,1,1,1)$  & 3 \\
  \hline
  $(1,1,1,1,1,1,0,0,1,1,0,0,1,1,1)$ & 3 & $(1,1,1,1,1,1,1,1,1,1,1,1,1,1,0)$  & 3 \\
  \hline
  $(1,1,1,1,1,1,0,0,1,1,1,1,0,0,1)$ & 3 & $(1,1,1,1,1,1,1,1,1,1,1,1,0,0,1)$  & 3 \\
  \hline
  $(1,1,1,1,1,1,0,0,1,1,1,1,1,1,0)$ & 3 & $(1,1,1,1,1,1,1,1,1,1,0,0,1,1,1)$  & 3 \\
  \hline
  $(1,1,1,1,1,1,0,0,1,1,1,1,1,1,1)$ & 3 & $(1,1,1,1,1,1,1,1,0,0,1,1,1,1,1)$  & 3 \\
  \hline
                                      &   & $(1,1,1,1,1,1,1,1,1,1,1,1,1,1,1)$  & 4
\end{tabular}}
{{\rm Table 2} \quad Diameters of io-decomposable Riordan graphs of
the Bell type with degree 16 such that the first 6 entries of its
$A$-sequence are all 1s}
\end{center}

By Lemma \ref{lemma-diameter-inequality}, using the results in Table
1 and 2 we obtain the following theorem.

\begin{theorem}
For $k\ge4$, let $G_{2^k}=G_{2^k}(g,zg)$ be an io-decomposable
Riordan graph and $G_{2^k}\not\cong CG_{2^k}$. If the first 16
entries in the binary $A$-sequence of $(g,zg)$ are not all 1s then
\begin{align*}
\textrm{diam}(G_{2^k})<\textrm{diam}(CG_{2^k})=k.
\end{align*}
\end{theorem}


\end{document}